\documentclass[12pt,twoside]{amsart}

\usepackage[textwidth=13.5cm]{geometry}
\usepackage[english]{babel}
\usepackage{graphicx}
\usepackage{amsmath}
\usepackage{amsfonts}

\begin{document}

\newcommand{\wk}{\mbox{$\,<$\hspace{-5pt}\footnotesize )$\,$}}

\numberwithin{equation}{section}
\newtheorem{teo}{Theorem}
\newtheorem{lemma}{Lemma}
\newtheorem{defi}{Definition}
\newtheorem{coro}{Corollary}
\newtheorem{prop}{Proposition}
\theoremstyle{remark}
\newtheorem{remark}{Remark}
\newtheorem{scho}{Scholium}
\numberwithin{lemma}{section}
\numberwithin{prop}{section}
\numberwithin{teo}{section}
\numberwithin{defi}{section}
\numberwithin{coro}{section}
\numberwithin{figure}{section}
\numberwithin{remark}{section}
\numberwithin{scho}{section}
\title{Optimal constants in normed planes}

\author{Vitor Balestro, Horst Martini, and Ralph Teixeira}
\address [V. Balestro] {CEFET/RJ Campus Nova Friburgo - Nova Friburgo - Brazil}
\email{vitorbalestro@id.uff.br}
\address [H. Martini] {Fakult\"at f\"ur Mathematik - Technische Universit\"at Chemnitz - 09107 Chemnitz - Germany;
\newline
\indent Dept. of Applied Mathematics - Harbin University of Science and Technology - 150080 Harbin - China}
\email{martini@mathematik.tu-chemnitz.de}
\address [R. Teixeira] {Instituto de Matem\'{a}tica e Estat\'{i}stica - UFF - Niter\'{o}i - Brazil}
\email{ralph@mat.uff.br}

\begin{abstract} We define new geometric constants for normed planes, determine their optimal values, and characterize types of planes for which these optimal values are attained. Relations of these constants to several topics, such as areas and distances from points to sides of triangles, are also investigated. We perform some calculations of new and old constants given by trigonometric functions for certain classes of norms, and a conjecture on a certain geometric constant is disproved.
\end{abstract}


\maketitle

\section{Introduction} \label{secintro}

When studying geometric properties of an arbitrary normed space, one may ask ``how different'' the geometry of such a space to Euclidean geometry is or, in the planar situation, how strong it differs for example from some Radon norm. This motivates the definition of geometric constants which can quantify somehow this difference. We may cite, for example, the \emph{James constant}
\begin{align*}
J(V) = \sup\{\min(||x+y||,||x-y||):x,y \in S\},
\end{align*}
where $S$ is the unit sphere of the normed space $(V,||\cdot||)$ (see \cite{gao}), the \emph{von Neumann-Jordan constant}
\begin{align*}
C_{NJ}(V) = \sup\left\{\frac{||x+y||^2+||x-y||^2}{2\left(||x||^2+||y||^2\right)}:x,y \in V\setminus\{0\}\right\},
\end{align*}
which relates to the parallelogram law for norms derived from inner products (cf. \cite{clarkson}), and the $D$ \emph{constant}
\begin{align*} D(V) = \inf\left\{\inf_{t\in\mathbb{R}}||x+ty||:x,y \in S \ \mathrm{and} \ x\dashv_Iy\right\}, \end{align*}
defined in \cite{wu} to estimate the difference between isosceles orthogonality and Birkhoff orthogonality.
 Constants to measure the difference between Roberts orthogonality and Birkhoff orthogonality were investigated in \cite{papini} and \cite{bmt2}. There is a vast literature regarding geometric constants in Banach spaces. We refer the reader, beyond the already cited papers, to \cite{alonso2}, \cite{clarkson}, \cite{clarkson2}, \cite{gao}, \cite{james}, \cite{kato}, \cite{maligranda1}, \cite{nikolova}, \cite{taka} and \cite{taka2}.

In \cite{szostok} a sine function for normed planes was defined, which motivated the definition of some new geometric constants. Two of them were briefly studied in \cite{bmt} (namely, the $c_E$ \textit{constant}, based on a type of polar coordinates for normed planes, and the $c_R$ \textit{constant}, which estimates how far a normed plane is from being Radon). The objective of this paper is to introduce and investigate other constants based on the sine function, to calculate their values for some $l_p$ planes and also for the normed planes whose unit circles are affine regular $(4n)$-gons. We also give a negative answer to a conjecture posed in \cite{wu} and referring to vectors which are
isosceles orthogonal.

The first new constant is defined in Section \ref{seccd}, and it quantifies how non-Radon a norm is. We study its optimal values, completely characterizing those planes for which they are attained. In Section \ref{secct} we study the variation of the possible values for the sine of an angle of an equilateral triangle in a Radon plane. We also provide upper and lower bounds as well as characterizations of planes which attain these values. This constant also relates to areas and to distances from interior points to the sides of an equilateral triangle. Sections \ref{secdpolygon} and \ref{seclp} are devoted to calculate trigonometric constants for regular $(4n)$-gonal planes and $l_p$ planes, respectively.

Throughout the text, $(V,||\cdot||)$ will always denote a \emph{Minkowski} or \emph{normed plane}. Its \emph{unit ball} and \emph{unit circle} are denoted by $B:= \{x : \|x\| \le 1\}$ and
 $S: = \{x : \|x\| = 1\}$, respectively. A nondegenerate symplectic bilinear form in $V$ will be denoted by $[\cdot,\cdot]$, and $||\cdot||_a$ stands for the associated \textit{antinorm}, i.e., for $x \in V$ we
  have $\|x\|_a = \sup \{|[y,x]| : y \in S\}$ (see \cite{martiniantinorms}). As usual, we say that a vector $x \in V$ is \textit{Birkhoff orthogonal} to a vector $y \in V$ (and denote this by $x \dashv_B y$) when $||x+ty||\geq ||x||$ for all $t \in \mathbb{R}$, and we call two vectors $x,y \in V$ \textit{isosceles orthogonal} whenever $||x+y|| = ||x-y||$. The \textit{sine function} $s:S\times S\rightarrow \mathbb{R}$ is defined as $s(x,y) = \inf_{t\in\mathbb{R}}||x+ty||$ (see \cite{bmt}). Also, $[ab]$ and $\left<ab\right>$ stand, respectively, for the \textit{segment} joining $a$ and $b$ and for the \textit{line} spanned by $a$ and $b$. The higher dimensional analogue of a normed plane, namely a finite dimensional real Banach space, is called a \textit{Minkowski space}. For basic concepts 
  and further background regarding the geometry of Minkowski planes and spaces we refer the reader to the book \cite{thompson} and to the surveys \cite{martini1} and \cite{martini2}, orthogonality types are discussed in \cite{alonso}.

\section{The constant $c_D$}\label{seccd}

Two directions $x,y \in V$ are said to be \textit{conjugate} if they are mutually Birkhoff orthogonal, i.e., if $x \dashv_B y$ and $y \dashv_B x$. It is well known that if any Birkhoff orthogonal pair of directions is conjugate (i.e., if Birkhoff orthogonality is a symmetric relation), then the plane is Radon and vice versa (see \cite{martiniantinorms} for many references and some background). It is proved in \cite{martini1} that every normed plane contains at least one pair of conjugate directions. On the other hand, in any non-Radon plane we must have a pair of directions $x,y \in V$ such that $x \dashv_B y$, but the converse does not hold. Hence, a way to estimate how far a plane is from being Radon is to measure somehow the worst case of those planes where $x$ is Birkhoff orthogonal to $y$, but $y$ is not Birkhoff orthogonal to $x$. We define the constant
\begin{align*} c_D(||\cdot||) := \inf\{s(y,x):x,y \in S \ \mathrm{and} \ x \dashv_B y\}. \end{align*}
Notice that $c_D$ is in fact attained for a pair $x,y \in S$. This comes immediately from the compactness of the set $\{(y,x) \in S\times S: x \dashv_B y\}$ and from the continuity of the sine function $s$. \\

Another trigonometric constant, for quantifying how far a plane is from being Radon, was defined in \cite{bmt}. The fact that the sine function is symmetric only for Radon planes motivated the definition of the constant
\begin{align*} c_R(||\cdot||) := \sup_{x,y\in S}|s(x,y)-s(y,x)|. \end{align*}
We start our investigation of the new constant $c_D$ by establishing an inequality involving both constants.\\
\begin{lemma}\label{ineqcdcr} In any Minkowski plane we have $c_R(||\cdot||) + c_D(||\cdot||)\geq 1$. Equality holds if and only if the constant $c_R$ is attained for a pair $x,y \in S$ such that $x \dashv_B y$.
\end{lemma}
\begin{proof} If $x \dashv_B y$, then $|s(x,y) - s(y,x)| = 1 - s(y,x)$. Hence, assuming that $s(y,x) = c_D(||\cdot||)$, we have
\begin{align*} c_R(||\cdot||) \geq |s(x,y) - s(y,x)| = 1 - s(y,x) = 1 - c_D(||\cdot||), \end{align*}
and the inequality follows. Assume now that equality holds. Then let $x,y \in S$ be such that $x \dashv_B y$ and $s(y,x) = c_D(||\cdot||)$. Thus, $c_R(||\cdot||) = 1 - c_D(||\cdot||) = s(x,y) - s(y,x)$, and this gives the desired implication. For the converse, assume that $x,y \in S$ are such that $x \dashv_B y$ and $c_R(||\cdot||) = s(x,y) - s(y,x) = 1 - s(y,x)$. We shall prove that $c_D(||\cdot||) = s(y,x)$. Indeed, let $z,w \in S$ be such that $z \dashv_B w$. Then
\begin{align*} 1-s(y,x) = c_R(||\cdot||) \geq s(z,w) - s(w,z) = 1 - s(w,z), \end{align*}
and hence $s(y,x) \leq s(w,z)$. What we want to show follows now from the definition of $c_D$.

\end{proof}

\noindent\textbf{Remark.} It is easy to see that $c_D(||\cdot||) = c_D\left(||\cdot||_a\right)$ for any normed plane $(V,||\cdot||)$. This is a consequence of the fact that $s(x,y) = s_a(y,x)$ for any $x,y \in S$, where $s_a$ denotes the sine function of the antinorm (see \cite{bmt}). Thus, in the sense of the constant $c_D$, a plane and its dual are non-Radon in the same measure.\\

Lemma \ref{ineqcdcr} has two interesting consequences. We start with the first one, which is to give upper and lower bounds for $c_D$ and characterizing also those planes for which they are attained. The second one is a calculation of $c_D$ for those planes whose unit circle is an affine regular $(4n)$-gon, and this will be presented in Section \ref{secdpolygon}.

\begin{prop}\label{cdbounds} Let $(V,||\cdot||)$ be a normed plane. We have
\begin{align*} \frac{1}{2} \leq c_D(||\cdot||) \leq 1, \end{align*}
and equality on the left side occurs if and only if the plane is rectilinear, and on the right side if and only if the plane is Radon.
\end{prop}
\begin{proof} The inequalities and equality on the left side follow immediately from Lemma \ref{ineqcdcr} and Theorem 5.1 in \cite{bmt}, which states that $c_R(||\cdot||) \leq \frac{1}{2}$ for any Minkowski plane, with equality if and only if it is rectilinear. For the other equality we just need to notice that if $c_D(||\cdot||) = 1$, then Birkhoff orthogonality is symmetric (see Lemma 2.2 in \cite{bmt}), and thus the plane is Radon.

\end{proof}

\section{The constants $c_t$ and $c_T$} \label{secct}

Let $(V,||\cdot||)$ be a Radon plane, and let $\Delta$ be an equilateral triangle such that its sides are the vectors $x,y$, and $z$. In \cite{bmt} it is proven that $s(x,y) = s(y,z) = s(z,x)$ and hence we may denote this number by $s(\Delta)$. We define the constants $c_t$ and $c_T$ to be the infimum and the supremum of the set $R_{||\cdot||} := \{s(\Delta): \Delta \subseteq (V,||\cdot||) \ \mathrm{is} \ \mathrm{equilateral} \}$, respectively. \\

Notice that, since the sine function is homogeneous, we have that the infimum and the supremum may be taken over the equilateral triangles with unit sides. In order to study these constants, we begin with equilateral triangles in Radon planes. Our main objective is to prove that in any Radon plane the constants $c_t$ and $c_T$ are indeed attained for some triangles.

\begin{prop}[\cite{martini1}, Proposition 33]\label{prop2} Given any segment $[pq]$ in a Minkowski plane, and a half-plane determined by the line $\left<pq\right>$. Then there exists a point $r$ in the half-plane such that $\Delta pqr$ is an equilateral triangle. The point $r$ is unique if and only if the longest segment in the unit circle parallel to $\left<pq\right>$ has length at most 1.
\end{prop}

\begin{lemma}\label{lemma1} Let $[pq] \subseteq S$ be a segment of the unit circle S of a Radon plane $(V,||\cdot||)$. Then $||p-q|| \leq 1$, and equality implies that $S$ is an affine regular hexagon.
\end{lemma}
\begin{proof} Assume that we have a segment $[pq] \subseteq S$ with $||p-q|| > 1$. Then we may take a point $r$ in the relative interior of $[pq]$ such that the triangle $\Delta opr$ is equilateral. Hence the affine regular hexagon whose vertices are $\pm p$, $\pm r$, and $\pm (p-r)$ is inscribed in the unit circle $S$. Notice that $q \dashv_B (p-r)$, and then $(p-r) \dashv_B q$ (the norm is Radon). But the line $l_1:t\mapsto (p-r) + tq$ intersects the segment $[pr]$ in a point different from $p$ (see Figure 3.1), and this contradiction yields $||p - q|| \leq 1$.\\
\begin{figure}[h]
\centering
\includegraphics[scale=0.66]{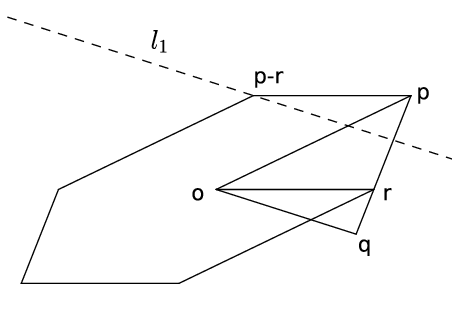}
\label{fig69}
\caption{$l_1$ does not support $S$}
\end{figure}

Assume now that we have a segment $[pq] \subseteq S$ with $||p-q|| = 1$. It follows that the affine regular hexagon $H$ with vertices $\pm p$, $\pm q$ and $\pm (p-q)$ is inscribed in $S$. Moreover, since $q \dashv_B (p-q)$ and
$V (\|\cdot\|)$ is Radon, it follows that $(p-q) \dashv_B q$. Hence the segment $[(p-q)p]$ is contained in the unit circle. Noticing that $p \dashv_B (p-q)$ and repeating the argument, we get $[(p-q)(-q)] \subseteq S$. This
implies $S = H$.

\end{proof}

\begin{coro}\label{coro2} Let $(V,||\cdot||)$ be a Radon plane endowed with a nondegenerate symplectic bilinear form $[\cdot,\cdot]$. Then for each $x \in S$ there exists a unique $y \in S$ such that $\Delta oxy$ is an equilateral triangle and $[x,y] > 0$.
\end{coro}
\begin{proof} Proposition \ref{prop2} and Lemma \ref{lemma1} guarantee that for a given $x\in S$ there exist exactly two points $z,y \in S$ such that $\Delta oxy$ and $\Delta oxz$ are equilateral triangles, each of them lying in one of the half-planes determined by the line $\left<ox\right>$. Since $[x,y]$ and $[x,z]$ have distinct signs, we have the result.

\end{proof}

This corollary allows us to define the map $e:S \rightarrow S$ which associates each $x \in S$ to the point $e(x) \in S$ such that $\Delta oxe(x)$ is equilateral and $[x,e(x)] > 0$. Now we define the function $f_{||\cdot||}:S \rightarrow \mathbb{R}$ setting $f_{||\cdot||}(x) = s(x,e(x))$. Obviously, $e$ and $f_{||\cdot||}$ are continuous and $R_{||\cdot||} = \mathrm{Im}\left(f_{||\cdot||}\right)$. Thus,
\begin{align*} c_t(||\cdot||) = \inf_{x\in S}f_{||\cdot||}(x) \ \mathrm{and} \end{align*}
\begin{align*}c_T(||\cdot||) = \sup_{x\in S}f_{||\cdot||}(x). \end{align*}
By the compactness of $S$, $c_t(||\cdot||)$ and $c_T(||\cdot||)$ are always attained for some equilateral triangles in $(V,\|\cdot\|)$.\\

\noindent\textbf{Remark.} It is easy to see that if $[pq]$ is any chord of the unit circle with $||p-q|| = 1$, then the affine regular hexagon with vertices $\pm p$, $\pm q $ and $\pm (p-q)$ is inscribed in S (\cite{martini1}, Proposition 34). As a consequence it follows that any orbit of the application $e$ has precisely six points. Even better, we have that $e^3(x) = -x$ for every $x \in S$.\\

The next proposition states that these trigonometric constants relate to the sum of the distances from internal points to the sides of an equilateral triangle.

\begin{prop}\label{prop1} Let $\Delta$ be an equilateral triangle with unit sides in a Radon plane $(V,||\cdot||)$. Then $s(\Delta)$ equals the sum of the distances of any internal point of $\Delta$ to its sides.
\end{prop}
\begin{proof} A triangle which is equilateral in the antinorm is called \textit{anti-equilateral}. Viviani's theorem (see \cite{martiniantinorms}, Corollary 8) states that the sum of the distances from any point inside an anti-equilateral triangle to its sides is constant. Obviously, when dealing with Radon planes a triangle is anti-equilateral if and only if it is equilateral, and thus the sum of the distances from any point inside an equilateral triangle in a Radon plane to its sides is constant. Let $\Delta \subseteq V$ be an equilateral triangle whose sides are given by the unit vectors $x, y$, and $z$, and let $p$ be a point in the interior of $\Delta$. Assume that the distances from $p$ to the sides of $\Delta$ are $d_1$, $d_2$, and $d_3$, respectively. Denoting by $A(\Delta)$ the area of the triangle, it follows (see \cite{martiniantinorms}, Proposition 2) that
\begin{align*} A(\Delta) = \frac{1}{2}\left(d_1+d_2+d_3\right). \end{align*}
On the other hand, if the symplectic bilinear form $[\cdot,\cdot]$ is scaled in such a way that $||\cdot|| = ||\cdot||_a$, then we have that $A(\Delta) = \frac{1}{2}s(\Delta)$ (\cite{bmt}, Section 3).
The desired then follows.

\end{proof}

\begin{coro}\label{coro1} The constants $c_t$ and $c_T$ present the minimum and the maximum possible value for the sum of distances from a point inside an equilateral triangle in $(V,||\cdot||)$ to its sides,
respectively.
\end{coro}

Since we always have $s(x,y) \leq 1$, clearly $c_T(||\cdot||) \leq 1$ holds for any Radon norm $||\cdot||$. Moreover, the equality $s(x,y) = 1$ holds if and only if $x \dashv_B y$, and hence the existence of an equilateral triangle $\Delta$ such that $s(\Delta) = 1$ is related to the existence of an equilateral triangle whose sides lie in mutually Birkhoff orthogonal directions. We will prove now that $c_T(||\cdot||) = 1$ if and only if the unit circle of the norm $||\cdot||$ is an affine regular hexagon.

\begin{prop}\label{prop3} The unit circle $S$ of a normed plane $\left(V,||\cdot||\right)$ is an affine regular hexagon if and only if there exist three distinct directions $x,y,z \in V$ such that
$x \dashv_B y$, $y \dashv_B z$, and $z \dashv_B x$.
\end{prop}
\begin{proof} Assume first that the unit circle of $V(\|\cdot\|)$ is an affine regular hexagon. Then there exist linearly independent unit vectors $x,y \in V$ such that its vertices are $\pm x$, $\pm y$, and $\pm (x+y)$. Hence, since the supporting lines of the sides are supporting lines of the unit ball through the vertices, we have immediately the demanded orthogonality relations (see Figure \ref{fig24}). \\

\begin{figure}[h]
\centering
\includegraphics{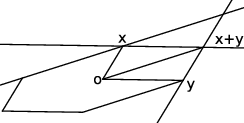}
\caption{Proposition \ref{prop3}}
\label{fig24}
\end{figure}

For the converse, let $x,y,z \in V$ be three distinct directions such that $x \dashv_B y$, $y \dashv_B z$, and $z \dashv_B x$. By the homogeneity of Birkhoff orthogonality we may assume that $||x|| = ||y|| = ||z|| = 1$. We write $z = \theta x + \sigma y$ for non-zero $\theta,\sigma \in \mathbb{R}$, and we shall prove first that $|\theta| = |\sigma| = 1$. The orthogonality relations give 
\begin{align*} 1 \leq ||x + \lambda y|| \ \mathrm{for \ all} \ \lambda \in \mathbb{R}\,, \end{align*}
\begin{align*} 1 \leq ||y + \beta z|| \ \mathrm{for \ all} \ \beta \in \mathbb{R}\,, \ \mathrm{and} \end{align*}
\begin{align*} 1 \leq ||z + \alpha x|| \ \mathrm{for \ all} \ \alpha \in \mathbb{R}. \end{align*}  
Setting $\lambda = \sigma/\theta$ in the first, $\beta = -1/\sigma$ in the second, and $\alpha = -\theta$ in the third case, we have
\begin{align*} 1 \leq \left|\left|x + \frac{\sigma}{\theta}y\right|\right| = \frac{1}{|\theta|}||z|| = \frac{1}{|\theta|}\,, \end{align*}
\begin{align*} 1 \leq \left|\left|-\frac{\theta}{\sigma}x\right|\right| = \left|\frac{\theta}{\sigma}\right|\,, \ \mathrm{and} \end{align*}
\begin{align*} 1 \leq ||\sigma y|| = |\sigma|. \end{align*}
Hence $|\sigma| \geq 1 \geq |\theta| \geq |\sigma|$, and this yields $|\theta| = |\sigma| = 1$. Let now $z = x+y$ (the other cases are analogous). We have that $x,y$, and $x+y$ are points of the unit circle. Therefore, the orthogonality $x \dashv_B y$ guarantees that the segment $[x(x+y)]$ is contained in $S$; the relation $y \dashv_B (x+y)$ gives $[(-x)y] \subseteq S$; and $(x+y) \dashv_B x$ yields that $S$ contains the segment $[y(x+y)]$. Thus, $S$ is the affine regular hexagon whose vertices are $\pm x$, $\pm y$, and $\pm (x+y)$.

\end{proof}

\noindent\textbf{Remark.} Notice that in the last proposition we did not require that $(V,||\cdot||)$ is Radon, although we are working with Radon planes within this section. This characterization of norms whose unit circle is an affine regular hexagon by using only Birkhoff orthogonality is interesting for itself (not depending on the trigonometric theory).

\begin{teo}\label{teo1} Let $(V,||\cdot||)$ be a normed plane. The following statements are equivalent: \\

\noindent\normalfont\textbf{(a)} \textit{The plane $(V,\|\cdot\|)$ is Radon and $c_T(||\cdot||) = 1$.} \\

\noindent\textbf{(b)} \textit{There exist three distinct directions $x,y,z \in V$ such that $x \dashv_B y$, $y\dashv_B z$, and $z \dashv_B x$.} \\

\noindent\textbf{(c)} \textit{The unit circle is an affine regular hexagon.}\\

\noindent\textbf{(d)} \textit{There exists an equilateral triangle with unit sides for which the sum of distances from any interior point to its sides equals $1$.}  \\
\end{teo}
\begin{proof} The bi-implication \textbf{(b)} $\Leftrightarrow$ \textbf{(c)} is Proposition \ref{prop3}. If \textbf{(a)} holds, then there exists an equilateral triangle with sides lying in directions $x,y$, and $z$ such that $s(x,y) = s(y,z) = s(z,x) = 1$, and hence we have \textbf{(b)} and then \textbf{(c)}. The implications \textbf{(c)} $\Rightarrow$ \textbf{(a)} and \textbf{(a)} $\Rightarrow$ \textbf{(d)} come easily, and thus it just remains to prove that \textbf{(d)} $\Rightarrow$ \textbf{(a)}. For this, recall that the distance of a point $p$ to a line $l$ in the direction $q$ is attained for a segment $[pr]$, $r \in l$, such that $(p-r) \dashv_B q$. The equilateral triangle described in \textbf{(d)} is such that the distance from each vertex to the supporting line of the respective opposite side is attained by the other two sides. Hence, if its sides lie in the directions $x,y$, and $z$, it follows in particular that $x \dashv_B y$, $y \dashv_B z$, and $z \dashv_B x$. This finishes the proof.

\end{proof}

In a Radon plane we may choose a nondegenerate symplectic bilinear form $[\cdot,\cdot]$ such that the associated antinorm coincides with the norm (see \cite{martiniantinorms}). This bilinear form yields an area unit in the plane if we set the area of the parallelogram spanned by the vectors $x,y$ to be $|[x,y]|$. Using an area argument, we may study the minimum possible value for $c_t(||\cdot||)$. It is interesting to notice that this minimum value also characterizes norms whose unit circle is an affine regular hexagon.\\

\begin{teo} \label{teo2} In any Radon plane $(V,||\cdot||)$ we have $c_t(||\cdot||) \geq \frac{3}{4}$ with equality if and only if the unit circle is an affine regular hexagon.
\end{teo}
\begin{proof} We use two inequalities given in \cite{martini1}. The first one (Proposition 51 in \cite{martini1}) states that given an equilateral triangle $\Delta \subseteq V$ with unit sides and area $A(\Delta)$, and denoting the area enclosed by the unit circle by $A(S)$, one has $A(S) \leq 8A(\Delta)$. The second (Proposition 50 in \cite{martini1}) gives that $A(S) \geq \frac{3}{4}A(P)$, where $A(P)$ is the area of a parallelogram of minimum area circumscribed about the unit circle (see also \cite{ms}). Since $(V,||\cdot||)$ is Radon, we have $A(\Delta) = \frac{1}{2}s(\Delta)$ (we refer the reader to \cite{bmt}, Section 3) and $A(P) = 4$. Then we have 
\begin{align*} 3 \leq A(S) \leq 4s(\Delta).
\end{align*}
The inequality $c_t(||\cdot||) \geq \frac{3}{4}$ is immediate. If equality holds, then, by compactness, we must have an equilateral triangle $\Delta$ with unit sides such that $s(\Delta) = \frac{3}{4}$. Thus $\frac{A(S)}{A(P)} = \frac{3}{4}$, and then, again by Proposition 50 in \cite{martini1}, the unit circle is an affine regular hexagon. This finishes the proof.

\end{proof}

For more on area of inscribed affine regular hexagons we refer the reader to \cite{lassak} and \cite{zeng}.

\begin{coro} A Radon plane $(V,||\cdot||)$ contains an equilateral triangle with unit sides such that the sum of the distances from each of its interior (or boundary) points to its sides is $\frac{3}{4}$ if and only if the unit circle is an affine regular hexagon.
\end{coro}

\noindent\textbf{Remark.} The reader may wonder if the equality $c_t(||\cdot||) = c_T(||\cdot||)$ implies that $(V,||\cdot||)$ is the Euclidean plane. This might be difficult to decide, and it is equivalent to say that all affine regular hexagons inscribed in $S$ have the same area. This would answer partially a question posed by Lassak in \cite{lassak}: which centrally symmetric closed convex curves have the property that all inscribed regular hexagons have the same area.

\section*{Interlude 1: a curious difference between regular $(4n)$-gonal and $(4n+2)$-gonal norms}

While regular $(4n+2)$-gons are always Radon curves (see \cite{heil} and \cite{martiniantinorms}), the regular $(4n)$-gons are \textit{equiframed curves} (see \cite{ms}), which are centrally symmetric closed convex curves with the property that each of their points is touched by a circumscribed parallelogram of smallest area. In \cite{ms}, Martini and Swanepoel argued that, in some sense, Radon curves and equiframed curves behave dually.  Namely, Radon curves can be defined as the centrally symmetric closed convex curves each point of which is a vertex of an inscribed parallelogram of largest area. We establish now another kind of ``dual behavior" of regular $(4n)$-gons and $(4n+2)$-gons: in some sense, regular $(4n)$-gons preserve Euclidean central angles, while regular $(4n+2)$-gons preserve Euclidean side lengths.

\begin{prop}\label{proppoly} Let $(V,||\cdot||)$ be a normed plane whose unit circle is an affine regular $(2n)$-gon, with $n \geq 3$. Then, if $x,y \in S$ are two consecutive vertices, we have\\

\normalfont\textbf{(a)} $||x - y||  = \displaystyle\sqrt{2-2\cos\frac{\pi}{n}}$ \textit{if and only if $n$ is odd, and}\\

\textbf{(b)} $s(x,y) = \sin\displaystyle\frac{\pi}{n}$\textit{ if and only if $n$ is even.}\\

\noindent Recall that these values present the side length of the regular $(2n)$-gon inscribed in the Euclidean unit circle, and the sine of the central angle of the same polygon, respectively.
\end{prop}
\begin{proof} Assertion \textbf{(a)} follows from the fact that any side of an affine regular $(4n+2)$-gon is parallel to a diameter which connects two of its vertices. \\

For \textbf{(b)}, let $a_1a_2...a_{4n}$ ($n \geq 2$) be the regular $(4n)$-gon which is the unit circle $S$ of $(V,||\cdot||)$. We will calculate $s(a_2,a_1)$. Notice that the direction $a_1$ supports the unit circle at the vertex $a_{n+1}$. Hence $s(a_2,a_1)$ is the length of the segment whose endpoints are the origin $o$ and the intersection $p$ of the segment $[oa_{n+1}]$ and the line $l: t \mapsto a_2 + ta_1$. To perform the calculations, assume that $S$ is the usual regular $(4n)$-gon inscribed in the Euclidean unit circle. Then $s(a_2,a_1)$ is the Euclidean length of the segment $[op]$. Since $l$ is (in the Euclidean sense) orthogonal to the segment $[oa_{n+1}]$ and the Euclidean measure of the angle $\wk oa_2p$ is $\frac{\pi}{2n}$, we have the desired value. Notice that $s(a_1,a_2) = s(a_2,a_1)$.

\end{proof}

\noindent\textbf{Remark.} Note that in the case where the unit circle is an affine regular $(4n+2)$-gon, we also have $s(x,y) = ||x-y||$ whenever $x,y \in S$ are consecutive vertices. This can be proved, for example, by using the Law of Sines studied in \cite{bmt}, Section 7.

\section{Calculations in regular $(4n)$-gonal norms}\label{secdpolygon}

In \cite{wu} the $D$ constant of a normed plane $(V,||\cdot||)$ is defined as
\begin{align*} D(V) = \inf\left\{\inf_{t\in\mathbb{R}}||x+ty||: x,y \in S \ \mathrm{and} \ x \dashv_I y \right\}. \end{align*}
Obviously, we may write $D(V) = \inf\left\{s(x,y): x,y \in S \ \mathrm{and} \ x\dashv_I y\right\}$, and hence this is a trigonometric constant. Our first task in this section is to calculate the $D$ constant for norms whose unit circles are affine regular $(4n)$-gons. These planes are not Radon (see \cite{heil}), but they are \textit{symmetric}, i.e., each of them has a pair $e_1,e_2 \in S$ such that
\begin{align*} ||e_1 + \lambda e_2|| = ||e_1 - \lambda e_2|| = ||e_2 + \lambda e_1|| = ||e_2 - \lambda e_1|| \end{align*}
for every $\lambda \in \mathbb{R}$. The vectors $e_1$ and $e_2$ are called a \textit{pair of axes} of $(V, \|\cdot\|)$. In \cite{wu} the reader may find a slightly more detailed study on symmetric normed planes, but we will only outline some of their properties which will be helpful later. The first result describes isosceles orthogonality in such spaces. Recall that in any normed plane, for a given $x \in S$ there exists only one (up to sign) $y \in S$ such that $x \dashv_I y$ (see \cite{alonso}).\\

\begin{teo}[\cite{wu}, Theorem 10]\label{teo3} Let $V$ be a symmetric Minkowski plane with a pair of axes $\{e_1,e_2\}$. If $x,y \in S$ are such that $x \dashv_I y$, then we have $x = \alpha e_1 + \beta e_2$, for some $\alpha,\beta \in \mathbb{R}$, if and only if $y = \pm (\beta e_1 - \alpha e_2)$.
\end{teo}

\begin{prop} \label{prop4} Let $(V,||\cdot||)$ be a symmetric Minkowski plane. If $x,y \in S$ and $x \dashv_I y$, then $s(x,y) = s(y,x)$.
\end{prop}
\begin{proof} Let $\{e_1,e_2\}$ be a pair of axes of $(V,||\cdot||)$. Theorem 9 in \cite{wu} states that, rescaling $[\cdot,\cdot]$ in such a way that $e_1,e_2 \in S_a$, we have that $\{e_1,e_2\}$ is also a pair of axes of the plane $(V,||\cdot||_a)$, and this is the chief ingredient of the proof. If $x,y \in S$ are isosceles orthogonal, then we may write, without loss of generality, $x = \alpha e_1 + \beta e_2$ and $y  = -\beta e_1 + \alpha e_2$ for some $\alpha, \beta \in \mathbb{R}$. Then
\begin{align*} s(x,y) = \frac{|[x,y]|}{||y||_a} = \frac{|[x,y]|}{||-\beta e_1 + \alpha e_2||_a} = \frac{|[x,y]|}{||\alpha e_1 + \beta e_2||_a} = \frac{|[y,x]|}{||x||_a} = s(y,x), \end{align*}
as we wished.

\end{proof}

\begin{teo}\label{teo4} Let $n \in \mathbb{N}$, and let $(V,||\cdot||_{4n})$ be a normed plane whose unit circle $S$ is given by an affine regular $(4n)$-gon. Then, denoting $D(V)$ by $D\left(||\cdot||_{4n}\right)$, we have
\begin{align*} D\left(||\cdot||_{4n}\right) = \frac{2\cos\frac{\pi}{4n}}{1+\cos\frac{\pi}{4n}}. \end{align*}

\end{teo}
\begin{proof} After an affine transformation we may assume that the unit circle is the usual regular Euclidean $(4n)$-gon with vertices $a_1,...,a_{4n}$, and that the pair of axes is the usual one presented by $e_1 = (1,0)$ and $e_2 = (0,1)$, lying in the respective midpoints of the sides $[a_{n+1}a_{n+2}]$ and $[a_1a_2]$. Clearly, it is enough to calculate $s(x,y)$ picking $x$ within the segment $[e_2a_2]$, and $y$ as the corresponding (up to sign) unit vector such that $x \dashv_I y$. From Theorem \ref{teo3} it follows that if $x \in [e_2a_2]$ is such that the (Euclidean) angle between the segments $[oe_2]$ and $[ox]$ is $\alpha$, then the corresponding $y$ is (up to sign) the point of the segment $[e_1a_{n+2}]$ such that the angle between $[oe_1]$ and $[oy]$ is also $\alpha$ (see Figure 4.1).

\begin{figure}[h]
\centering
\includegraphics{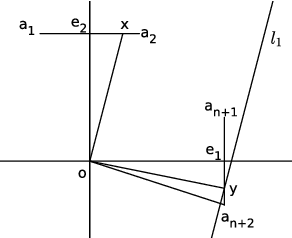}
\label{fig1t}
\caption{Calculating $D\left(||\cdot||_{4n}\right)$}
\end{figure}

Now, let $p$ be the intersection of the line $l_1: t \mapsto y + tx$ and the segment $[oa_{n+2}]$, and denote by $||\cdot||_E$ the usual Euclidean norm. Considering that $||a_{n+2}||_E = 1$, we have that the value of $s(y,x)$ is the Euclidean length of the segment $[op]$ (this follows since the direction $x$ supports the unit circle at the vertex $a_{n+2}$; see the geometric interpretation of the sine function presented in \cite{bmt}, Section 2). Let, then, $||p|| = d$. Hence $||a_{n+2} - p|| = 1 - d$. Notice that the line $l_1$ is orthogonal, in the Euclidean sense, to the segment $[oy]$. Hence basic triangle trigonometry in the triangles $\Delta oyp$ and $\Delta ypa_{n+2}$ yields
\begin{align*}\frac{d\sin\left(\frac{\pi}{4n} - \alpha\right)}{\cos\frac{\pi}{4n}} = \frac{1-d}{\sin\alpha}. \end{align*}
Thus,
\begin{align*} d = \frac{\cos\frac{\pi}{4n}}{\cos\frac{\pi}{4n} + \sin\left(\frac{\pi}{4n} - \alpha\right)\sin\alpha}. \end{align*}
Basic calculus shows that the function $\alpha \mapsto \sin\left(\frac{\pi}{4n}-\alpha\right)\sin\alpha$ attains its maximum over $[0,\frac{\pi}{4n}]$ at $\alpha = \frac{\pi}{8n}$. Then 
\begin{align*} s(y,x) = d \geq \frac{\cos\frac{\pi}{4n}}{\cos\frac{\pi}{4n} + \sin^2\frac{\pi}{8n}} = \frac{2\cos\frac{\pi}{4n}}{1+\cos\frac{\pi}{4n}}. \end{align*}
By Proposition \ref{prop4} we have that it is not necessary to evaluate $s(x,y)$. Indeed, $s(x,y) = s(y,x)$. The optimum value is attained for $x \in [e_2a_2]$ such that $\left.[ox\right>$ is the (Euclidean) bisector of the angle $\wk e_2oa_2$. This finishes the proof.

\end{proof}

\begin{coro}\label{coro3} We have $D\left(||\cdot||_{4n}\right) \rightarrow 1$ as $n \rightarrow \infty$.
\end{coro}

    Now we will calculate the value of the constant $c_E$ defined in \cite{bmt} for regular $(4n)$-gonal norms. Let us recall the definition. Given a pair $x,y \in S$ of mutually Birkhoff orthogonal unit vectors (i.e., a conjugate pair), we define
\begin{align*}c_E(x,y) := \sup_{z\in S}\left(s(z,x)^2+s(z,y)^2\right) - \inf_{z\in S}\left(s(z,x)^2+s(z,y)^2\right). \end{align*}
Then the constant $c_E(||\cdot||)$ is defined to be the supremum of $c_E(x,y)$ taken over all conjugate pairs of $(V,||\cdot||)$.

\begin{teo}\label{teocepoly} Let $(V,||\cdot||_{4n})$ be a plane endowed with a norm whose unit circle is an affine regular $(4n)$-gon, with $n \in \mathbb{N}$. Then
\begin{align*} c_E(||\cdot||_{4n}) = \left(\tan\frac{\pi}{4n}\right)^2. \end{align*}
\end{teo}
\begin{proof} Let the unit circle $S_{4n}$ be the polygon $a_1...a_{4n}$, and denote by $m_j$ the midpoint of the segment $[a_ja_{j+1}]$, identifying $a_{4n+1}$ with $a_1$. Then it is clear that any conjugate pair must be $\{m_j,m_{j+n}\}$ or $\{a_j,a_{j+n}\}$, for some $j$. Assume first that $S_{4n}$ is the regular $(4n)$-gon in the standard Euclidean plane $(\mathbb{R}^2,||\cdot||_E)$ for which the midpoints $m_1$ and $m_{n+1}$ are precisely the vectors $e_2 = (0,1)$ and $e_1 = (1,0)$. Hence for any $z \in S_{4n}$, the value of $s(z,m_1)^2 + s(z,m_{n+1})^2$ is precisely the squared Euclidean norm of $z$ (see \cite{bmt}, Section 2). Notice that the unit circle of $||\cdot||_E$ is the Euclidean circle inscribed to the polygon $S_{4n}$, and therefore it is easy to see that $\inf_{z\in S_{4n}}\left(s(z,m_1)^2 + s(z,m_{n+1})^2\right) = 1$, attained whenever $z$ is the midpoint of a side of $S_{4n}$. Also we have $\sup_{z\in S_{4n}}\left(s(z,m_1)^2+s(z,m_{n+1})^2\right) = 1 + \left(\tan\frac{\pi}{4n}\right)^2$, attained when $z$ is a vertex of $S_{4n}$. Thus,
\begin{align*} c_E(m_1,m_{n+1}) = \left(\tan\frac{\pi}{4n}\right)^2. \end{align*}
Now, let $S_{4n}$ be the regular $(4n)$-gon in the standard Euclidean plane such that the vertices $a_1$ and $a_{n+1}$ are $e_2$ and $e_1$, respectively. Again, the value of $s(z,a_1)^2 + s(z,a_{n+1})^2$ is the squared Euclidean norm of $z$, but now the unit circle of $||\cdot||_E$ is the Euclidean circle circumscribed about $S_{4n}$. Hence we have $\inf_{z\in S_{4n}}\left(s(z,a_1)^2+s(z,a_{n+1})^2\right) = \left(\cos\frac{\pi}{4n}\right)^2$ and $\sup_{z\in S_{4n}}\left(s(z,a_1)^2 + s(z,a_{n+1})^2\right) = 1$, both attained for the same cases as before. It follows that
\begin{align*} c_E(a_1,a_{n+1}) = 1 - \left(\cos\frac{\pi}{4n}\right)^2 = \left(\sin\frac{\pi}{4n}\right)^2, \end{align*}
and this concludes the proof, once $\tan\frac{\pi}{4n} > \sin\frac{\pi}{4n} > 0$ for every $n \in \mathbb{N}$.

\end{proof}

\begin{coro} We have $c_E(||\cdot||_{4n}) \rightarrow 0$ when $n \rightarrow \infty$.
\end{coro}

\noindent\textbf{Remark.} In some sense, the pairs of conjugate directions for which $c_E(x,y)$ attains its supremum and its infimum are ``natural choices" of pairs of axes. Namely, they are the bi-orthogonal pairs of directions for which the Pythagorean theorem has the largest and the smallest distortion, respectively, when compared with Euclidean geometry.\\

We finish this section by calculating the constant $c_D$ for regular $(4n)$-gonal norms, as mentioned before Proposition \ref{cdbounds}.

\begin{teo}\label{cd4n} Let $n \in \mathbb{N}$ and denote by $||\cdot||_{4n}$ a norm whose unit circle is an affine regular $(4n)$-gon. Then
\begin{align*} c_D\left(||\cdot||_{4n}\right) = \left(\cos\frac{\pi}{4n}\right)^2. \end{align*}

\end{teo}
\begin{proof} Theorem 5.2 in \cite{bmt} states that
\begin{align*} c_R\left(||\cdot||_{4n}\right) = \left(\sin\frac{\pi}{4n}\right)^2,\end{align*}
and in the proof it is clarified that this value is attained for a pair $x,y \in S$ such that $x \dashv_B y$. Then the result follows immediately from Lemma \ref{ineqcdcr}.

\end{proof}

\begin{coro} We have $c_D\left(||\cdot||_{4n}\right) \rightarrow 1$ as $n \rightarrow \infty$.
\end{coro}

\section*{Interlude 2: Discussing a conjecture from \cite{wu}}

In \cite{wu}, Remark 13, the authors conjectured the existence of a normed space $(V,||\cdot||)$ for which there are no $x,y \in S$ with $x \dashv_{I} y$ such that $D(V) = \inf_{t\in\mathbb{R}}||x+ty||$. We use the sine function approach to prove that such a finite dimensional space does not exist.

\begin{prop}\label{prop1} Let $(V,||\cdot||)$ be a Minkowski space. Then there exist $x,y \in S$ with $x \dashv_I y$ such that $D(V) = s(x,y)$.
\end{prop}
\begin{proof} The set $I \subseteq S\times S$ given by
\begin{align*} I = \{(x,y) \in S\times S: x \dashv_I y \} \end{align*}
is compact. In \cite{bmt} it is proved that the sine function $s:S\times S \rightarrow \mathbb{R}$ is continuous in planes, but the proof can obviously be extended to higher dimensional spaces. It follows that $D(V)$ is the infimum of the compact image $s(I)$, and the proof is finished.

\end{proof}

\section{Calculations in $l_p$ planes} \label{seclp}

We devote this section to the calculation of some trigonometric constants for $l_p$ planes. Here, as usual, the $l_p$ plane is the usual $\mathbb{R}^2$ endowed with the norm $||\cdot||_p$ defined as
\begin{align*} ||(\alpha,\beta)||_p = \left(|\alpha|^p+|\beta|^p\right)^{\frac{1}{p}}, \end{align*}
where $(\alpha,\beta) \in \mathbb{R}^2$ is the usual coordinate representation in the basis $\{e_1,e_2\}$, and $p \in [1,\infty]$. Also, denote the unit circle of $(\mathbb{R}^2,||\cdot||_p)$ by $S_p$. We start by calculating the constant $c_E$ for an $l_p$ plane.

\begin{teo}\label{teocelp} Let $(\mathbb{R}^2,||\cdot||_p)$ be an $l_p$ plane with $1 \leq p \leq \infty$. Then
\begin{align*}c_E(||\cdot||_p) = 2^{|1/q-1/p|} - 1, \end{align*}
where $q \in [1,\infty]$ is such that $\frac{1}{p} + \frac{1}{q} = 1$.
\end{teo}
\begin{proof} Assume that $1 < q < 2 < p < \infty$, and $\frac{1}{p}+\frac{1}{q} = 1$. The case $p = q = 2$ is the Euclidean plane, and the norms $||\cdot||_1$ and $||\cdot||_{\infty}$ are the rectilinear norms, for which $c_E$ was already calculated in Section \ref{secdpolygon}. We begin with the calculation of $c_E(||\cdot||_p)$. To perform the calculations, assume that the symplectic bilinear form $[\cdot,\cdot]$ is the usual determinant in $\mathbb{R}^2$. Thus, the associated antinorm is precisely the conjugate norm $||\cdot||_q$. It is easy to see that the only pairs of conjugate directions of $l_p$ and $l_q$ are $\{e_1,e_2\}$ and $\{e_1+e_2,e_2-e_1\}$. Therefore,
\begin{align*} c_E(||\cdot||_p) = \max\left\{c_{E,p}(e_1,e_2),c_{E,p}(e_1+e_2,e_2-e_1)\right\}, \end{align*}
where we introduce the notation $c_{E,p}$ to avoid confusion when performing the calculations for the $l_q$ norm. Let $z \in S_p$ be given by $z = \alpha e_1 + \beta e_2$. Then $|\alpha|^p + |\beta|^p = 1$, and
\begin{align*}
s(z,e_1)^2 + s(z,e_2)^2 = \alpha^2 + \beta^2 = \alpha^2 + \left(1-|\alpha|^p\right)^{2/p}.
\end{align*}
Hence, in order to determine $c_E(e_1,e_2)$, we just have to study the behavior of the function $\alpha \mapsto \alpha^2 + \left(1-|\alpha|^p\right)^{2/p}$ when $\alpha$ ranges from $0$ to $1$ (the symmetry guarantees that we do not need to consider negative values for $\alpha$). Geometrically, this is the squared Euclidean norm of the point $(\alpha,\beta) \in S_p$. Standard calculus gives immediately that the maximum and minimum values attained by this function are $2^{1-2/p}$ and $1$, respectively. Thus,
\begin{align*} c_E(e_1,e_2) = 2^{1-2/p} - 1. \end{align*}
On the other hand, we use the formula $s(x,y) = \frac{|[x,y]|}{||x||.||y||_a}$ to calculate
\begin{align*} s(z,e_1+e_2)^2 + s(z,e_2-e_1)^2 = \frac{(\alpha-\beta)^2}{||e_1+e_2||^2_q} + \frac{(\alpha + \beta)^2}{||e_2-e_1||^2_q} = 2^{1-2/q}\left(\alpha^2+\left(1-|\alpha|^p\right)^{1/p}\right). \end{align*}
Thus,
\begin{align*} c_{E,p}(e_1+e_2,e_2-e_1) = 2^{1-2/q}\left(2^{1-2/p} - 1\right). \end{align*}
Since $2^{1-2/q} < 1$, we have that
\begin{align*} c_E(||\cdot||_p) = 2^{1-2/p} - 1 = 2^{1/q - 1/p} - 1. \end{align*}
Now, we calculate $c_E(||\cdot||_q)$. The argument is basically the same, but the maximum and minimum values attained by $\alpha \mapsto \alpha^2 + \left(1-|\alpha|^q\right)^{2/q}$, when $\alpha$ ranges from $0$ to $1$, are $1$ and $2^{1-2/q}$, respectively. Hence
\begin{align*} c_{E,q}(e_1,e_2) = 1 - 2^{1-2/q}. \end{align*}
We calculate $c_{E,q}$ relative to the other pair of conjugate directions in the same way as before, to obtain
\begin{align*} c_{E,q}(e_1+e_2,e_2-e_1) = 2^{1-2/p}\left(1-2^{1-2/q}\right). \end{align*}
And since $2^{1-2/p} \geq 1$, we have
\begin{align*} c_E(||\cdot||_q) = 2^{1-2/p}\left(1-2^{1-2/q}\right) = 2^{1-2/p}-1 = 2^{1/q-1/p} - 1. \end{align*}
In particular, $c_E(||\cdot||_p) = c_E(||\cdot||_q)$. This means, in some sense, that if $p$ and $q$ are conjugate, then the spaces $l_p$ and $l_q$ are equally far from being Euclidean.

\end{proof}

We can also calculate, but not explicitly, the constant $c_D$ for $l_p$ planes. This is the subject of our last proposition.

\begin{prop}\label{calccd} Let $1 < q \leq 2 \leq p < \infty$ be numbers such that $\frac{1}{p} + \frac{1}{q} = 1$, and consider the norms $||\cdot||_p$ and $||\cdot||_q$ in $\mathbb{R}^2$ defined as usual. We have
\begin{align*} c_D\left(||\cdot||_p\right)=c_D\left(||\cdot||_q\right) = \\ = \inf_{\alpha \in \left.\left(0,\frac{1}{2}\right]\right.}\left(\alpha(1-\alpha)^{q-1} + (1-\alpha)\alpha^{q-1}\right)^{-\frac{1}{q}}\left(\alpha(1-\alpha)^{p-1}+(1-\alpha)\alpha^{p-1}\right)^{-\frac{1}{p}}
. \end{align*}
\end{prop}
\begin{proof} By symmetry it is enough to consider the points of $S_p$ of the form $x_{\alpha}=\left((1-\alpha)^{\frac{1}{p}},\alpha^{\frac{1}{p}}\right)$ with $\alpha \in \left[0,\frac{1}{2}\right]$. By simple differentiation we see that the supporting direction to $B_p$ at such a point is $y_{\alpha}=\left(-(1-\alpha)^{-\frac{1}{q}},\alpha^{-\frac{1}{q}}\right)$. Now we just apply the formula
\begin{align*} s\left(y_{\alpha},x_{\alpha}\right) = \frac{|[x_{\alpha},y_{\alpha}]|}{||x_{\alpha}||_q||y_{\alpha}||_p},\end{align*}
and the proposition follows after some simple calculations.

\end{proof}

\noindent\textbf{Acknowledgements.} The first author wants to thanks CAPES for financial support during the preparation of this manuscript.

\end{document}